\documentclass{article}%
\usepackage{amsmath}
\usepackage{amsfonts}
\usepackage{amssymb}
\usepackage{graphicx}
\usepackage{wrapfig}
\usepackage{float}
\usepackage{xcolor}%
\providecommand{\U}[1]{\protect\rule{.1in}{.1in}}
\graphicspath{ {./images/} }
\newtheorem{theorem}{Theorem}

\newtheorem{lemma}[theorem]{Lemma}

\newenvironment{proof}[1][Proof]{\noindent\textbf{#1.} }{\ \rule{0.5em}{0.5em}}
\numberwithin{equation}{section}
\begin{document}

\textbf{ON SOME INTERPOLATION INEQUALITIES  }

\textbf{BETWEEN  H\"{O}LDER AND LEBESGUE'S SPACES}

\bigskip

\textbf{S.P. Degtyarev }

\textbf{spdegt@mail.ru}

\bigskip
\bigskip

\textbf{Moscow Technical University of Communications and Informatics,}

\textbf{Moscow, Russian Federation}

\bigskip

\bigskip

\bigskip

{\small \noindent\textbf{Abstract. } We present a simple proof of some
interpolation inequalities between H\"{o}lder and Lebesgue's spaces. As an
example, to demonstrate the simplicity of their applications to nonlinear PDE,
we give also a simple proof of an a-priory estimate in a smooth H\"{o}lder
class for a solution to some quasilinear parabolic initial boundary value
problem. }


{\small \medskip}


\section{Introduction}

In this paper we deal with some simple interpolation inequalities between
H\"{o}lder and Lebesgue's spaces. The importance of interpolation in the
investigations of different problems for partial differential equations is
widely known, see, for example, \cite{Lun1}. \ The inequalities from this
paper can be used, for example, in some appropriate situations to obtain an
a-priory estimate in H\"{o}lder spaces when an estimate in a particular
Lebesgue's space is already obtained. And the last estimate can very often be
obtained more easily. We give a simple example of an application of the
interpolation inequalities from this paper to an a-priory estimate for a
solutions to some quasilinear parabolic initial boundary value problem.

The subject of interpolation inequalities is so vast that it is impossible to
describe even short history of this question, so we confine ourselves to
refereing books \cite{Lun1}, \cite{6}, \cite{Triebel} and the references
therein. Most of known interpolation inequalities, however, concern with the
interpolation of norms in different homogeneous scales of function spaces,
such as Lebesgue's spaces, Sobolev spaces or H\"{o}lder spaces. At the same
time the author of the present paper is not aware of some aggregate results on
the interpolation of norms between different scales of function spaces. Note,
nevertheless, that the present paper is motivated, in particular, by \cite{7},
where the simplest situation was considered, and where we got the idea of
obtaining such sharp inequalities. Note also that in this paper we consider
for simplicity the known functional spaces, designed for parabolic and
elliptic equations of the second order.

Let's now give several definitions and auxiliary facts. We are going to use
standard spaces $C^{m+\alpha}(\overline{\Omega})$ and $C^{m+\alpha
,\frac{m+\alpha}{2}}(\overline{\Omega}_{T})$ of H\"{o}lder continuous
functions $u(x)$ and $u(x,t)$, where $m=0,1,2,...$, $\alpha\in(0,1)$, $\Omega$
is a given domain in $R^{N}$ (bounded or unbounded) with smooth boundary (of
the class $C^{m+\alpha}$), $\Omega_{T}=\Omega\times(0,T)$, $\overline{\Omega
}_{T}=\overline{\Omega}\times\lbrack0,T]$, $T>0$ is a given constant. The
norms in the space $C^{m+\alpha}(\overline{\Omega})$ is defined by%

\begin{equation}
\left\Vert u\right\Vert _{C^{m+\alpha}(\overline{\Omega})}\equiv
|u|_{\overline{\Omega}}^{(m+\alpha)}=|u|_{\overline{\Omega}}^{(m)}%
+\left\langle u\right\rangle _{\overline{\Omega}}^{(m+\alpha)}, \label{1.5}%
\end{equation}
where%

\begin{equation}
|u|_{\overline{\Omega}}^{(m)}\equiv{\sum\limits_{|\overline{\beta}|\leq m}%
}\max_{\overline{\Omega}}\left\vert D_{x}^{\overline{\beta}}u(x)\right\vert
,\quad\left\langle u\right\rangle _{\overline{\Omega}}^{(m+\alpha)}\equiv
{\sum\limits_{|\overline{\beta}|=m}}\left\langle D_{x}^{\overline{\beta}%
}u(x)\right\rangle _{\overline{\Omega}}^{(\alpha)}, \label{1.6}%
\end{equation}

\[
\left\langle w(x)\right\rangle _{\overline{\Omega}}^{(\alpha)}\equiv
\left\langle w(x)\right\rangle _{x,\overline{\Omega}}^{(\alpha)}\equiv
\sup_{x,y\in\overline{\Omega}}\frac{|w(x)-w(y)|}{|x-y|^{\alpha}},
\]
$\overline{\beta}=(\beta_{1},\beta_{2}, ..., \beta_{N})$ is a multiindex,
$\beta_{i}=0,1,2,...$, $|\overline{\beta}|=\beta_{1}+\beta_{2}+...+\beta_{N}$,%

\[
D_{x}^{\overline{\beta}}u=\frac{\partial^{\beta_{1}}\partial^{\beta_{2}%
}...\partial^{\beta_{N}}u(x)}{\partial x_{1}^{\beta_{1}}\partial x_{2}%
^{\beta_{2}}...\partial x_{3}^{\beta_{N}}}.
\]
Analogously, the norm in the space $C^{m+\alpha,\frac{m+\alpha}{2}}%
(\overline{\Omega}_{T})$ is defined by%

\begin{equation}
\left\Vert u\right\Vert _{C^{m+\alpha,\frac{m+\alpha}{2}}(\overline{\Omega
}_{T})}\equiv|u|_{\overline{\Omega}_{T}}^{(m+\alpha)}=|u|_{\overline{\Omega
}_{T}}^{(m)}+\left\langle u\right\rangle _{\overline{\Omega}_{T}}^{(m+\alpha
)}, \label{1.7}%
\end{equation}

\begin{equation}
|u|_{\overline{\Omega}_{T}}^{(m)}\equiv{\sum\limits_{|\overline{\beta}|+2l\leq
m}}\max_{\overline{\Omega}_{T}}\left\vert D_{t}^{l}D_{x}^{\overline{\beta}%
}u(x,t)\right\vert ,\quad\left\langle u\right\rangle _{\overline{\Omega}_{T}%
}^{(m+\alpha)}=\left\langle u\right\rangle _{x,\overline{\Omega}_{T}%
}^{(m+\alpha)}+\left\langle u\right\rangle _{t,\overline{\Omega}_{T}%
}^{(m+\alpha)}, \label{1.8}%
\end{equation}

\[
\left\langle u\right\rangle _{x,\overline{\Omega}_{T}}^{(m+\alpha)}\equiv
{\sum\limits_{0\leq m-|\beta|-2l\leq1}}\left\langle D_{t}^{l}D_{x}%
^{\overline{\beta}}u(x)\right\rangle _{x,\overline{\Omega}}^{(\alpha)},
\]

\begin{equation}
\left\langle u\right\rangle _{t,\overline{\Omega}_{T}}^{(m+\alpha)}\equiv
{\sum\limits_{0\leq m-|\beta|-2l\leq1}}\left\langle D_{t}^{l}D_{x}%
^{\overline{\beta}}u(x)\right\rangle _{t,\overline{\Omega}}^{(\frac
{m-|\beta|-2l+\alpha}{2})}. \label{1.9}%
\end{equation}
Note that for the space $C^{m+\alpha,\frac{m+\alpha}{2}}(\overline{\Omega}%
_{T})$ we also have the estimate (see, for example,\cite{5})%

\begin{equation}
\left\langle u\right\rangle _{t,\overline{\Omega}_{T}}^{(m+\alpha)}%
={\sum\limits_{0\leq m-|\beta|-2l\leq1}}\left\langle D_{t}^{l}D_{x}%
^{\overline{\beta}}u(x)\right\rangle _{t,\overline{\Omega}}^{(\frac
{m-|\beta|-2l+\alpha}{2})}\leq\label{1.10}%
\end{equation}

\[
\leq C\left(  \left\langle u\right\rangle _{x,\overline{\Omega}_{T}%
}^{(m+\alpha)}+\left\langle D_{t}^{[m/2]}u(x)\right\rangle _{t,\overline
{\Omega}}^{(\frac{m-2[m/2]+\alpha}{2})}\right)  ,
\]
where here and everywhere below we denote by $C$, $\nu$ all absolute constants
or constants depending on fixed data of the problem. It is known (see
\cite{Triebel}, \cite{5} , \cite{Gol18}) that the seminorm $\left\langle
u\right\rangle _{\overline{\Omega}_{T}}^{(m+\alpha)}$ is equivalent to the seminorm%

\[
\left\langle u\right\rangle _{\overline{\Omega}_{T}}^{(m+\alpha)}\simeq
C_{k,l}\left(  \sup_{\overline{h}\in R^{3};x,x+\overline{h}\in\overline
{\Omega};t\in\lbrack0,T]}\frac{|\Delta_{x,\overline{h}}^{k}u(x,t)|}%
{|\overline{h}|^{m+\alpha}}+\right.
\]

\begin{equation}
\left.  +\sup_{\Delta t>0;x\in\overline{\Omega};t\in\lbrack0,T]}\frac
{|\Delta_{t,\Delta t}^{l}u(x,t)|}{|\Delta t|^{\frac{m+\alpha}{2}}}\right)  .
\label{1.10.1}%
\end{equation}
Here $k$ , $l$ are some fixed integers such that $k>m+\alpha$, $l>\frac
{m+\alpha}{2}$, $\Delta_{x,\overline{h}}^{k}u=\Delta_{x,\overline{h}}%
(\Delta_{x,\overline{h}}^{k-1}u)$, $\Delta_{x,\overline{h}}u=u(x+\overline
{h},t)-u(x,t)$ and analogously $\Delta_{t,\Delta t}^{l}u=\Delta_{t,\Delta
t}(\Delta_{t,\Delta t}^{l-1}u)$, $\Delta_{t,\Delta t}u=u(x,t+\Delta t)-u(x,t)$.

The above relations can be written also in a more concise way. Denote
$\overline{H}=(\overline{h},\Delta t)$ ,
\[
\left\Vert \overline{H}\right\Vert \equiv|\overline{h}|+|\Delta t|^{\frac
{1}{2}},
\]
and denote $\Delta_{\overline{H}}u(x,t)=u(x+\overline{h},t+\Delta t)-u(x,t)$.
Then \eqref{1.10.1} is equivalent to%

\begin{equation}
\left\langle u\right\rangle _{\overline{\Omega}_{T}}^{(m+\alpha)}\simeq
C_{k}\sup_{(x,t),(x,t)+\overline{H}\in\overline{\Omega}_{T}}\frac
{|\Delta_{\overline{H}}^{k}u(x,t)|}{||\overline{H}||^{m+\alpha}}.
\label{1.10.2}%
\end{equation}

We also use for functions $u(x)$ or $u(x,t)$ spaces $L_{p}(\Omega)$ or
$L_{p}(\Omega_{T})$ respectively with the norms $\left\Vert u\right\Vert
_{p,\Omega}$ and $\left\Vert u\right\Vert _{p,\Omega_{T}}$ correspondingly,
$p>1$.

For the spaces $C^{m+\alpha}(\overline{\Omega})=C^{l}(\overline{\Omega})$ with
noninteger $l=m+\alpha$ and also with integer $l\geq0$ we have the following
interpolation inequalities (see, for example, \cite{6})%

\begin{equation}
|u|_{\overline{\Omega}}^{(l)}\leq C\left(  |u|_{\overline{\Omega}}^{(l_{2}%
)}\right)  ^{\omega}\left(  |u|_{\overline{\Omega}}^{(l_{1})}\right)
^{1-\omega},\quad\omega=\frac{l-l_{1}}{l_{2}-l_{1}},\quad l_{1}<l<l_{2},
\label{2.1}%
\end{equation}
and analogous inequalities for anisotropic spaces $C^{l,\frac{l}{2}}
(\overline{\Omega}_{T})$%

\begin{equation}
|u|_{\overline{\Omega}_{T}}^{(l)}\leq C\left(  |u|_{\overline{\Omega}_{T}%
}^{(l_{2})}\right)  ^{\omega}\left(  |u|_{\overline{\Omega}_{T}}^{(l_{1}%
)}\right)  ^{1-\omega},\quad\omega=\frac{l-l_{1}}{l_{2}-l_{1}},\quad
l_{1}<l<l_{2}, \label{2.2}%
\end{equation}
where $l_{1},l_{2}$ may be either integer or noninteger.

The further content of the paper is as follows. In the next section of the
paper we prove some interpolation inequalities for functions from H\"{o}lder
spaces (in the case of unbounded domain we need the intersection of H\"{o}lder
and Lesbesgue's spaces). And in the last third section we apply these
inequalities to a-priory estimates and solvability of a model (just for
simplicity) problem for partial differential equations.

\section{Interpolation inequalities.}

We start with the following interpolation inequality as the key particular case.

\begin{lemma}
\label{L2.1}

Let $l>0$ be noninteger and let $u(x,t)\in C^{l,\frac{l}{2}}(\overline{\Omega
}_{T})\cap L_{p}(\Omega_{T})$. Then%
\begin{equation}
|u|_{\overline{\Omega}_{T}}^{(0)}\leq C\left(  |u|_{\overline{\Omega}_{T}%
}^{(l)}\right)  ^{\omega}\left(  \left\Vert u\right\Vert _{p,\Omega_{T}%
}\right)  ^{1-\omega},\quad\omega=\frac{N+2}{lp+N+2}. \label{2.3.1}%
\end{equation}
\newline

If for the function $u(x,t)$ the following parabolic norm is finite%
\begin{equation}
\sup_{0<t<T}\left\Vert u(\cdot,t)\right\Vert _{p,\Omega}<\infty, \label{2.3.2}%
\end{equation}
then%
\begin{equation}
|u|_{\overline{\Omega}_{T}}^{(0)}\leq C\left(  |u|_{\overline{\Omega}_{T}%
}^{(l)}\right)  ^{\sigma}\left(  \sup_{0<t<T}\left\Vert u(\cdot,t)\right\Vert
_{p,\Omega}\right)  ^{1-\sigma},\quad\sigma=\frac{N}{lp+N}. \label{2.3.3}%
\end{equation}

\end{lemma}

\begin{proof}
We are going to use relation \eqref{1.10.2}. Let first $\Omega=R^{N}$,
$T=\infty$. Let $x,y\in R^{N}$, $x$ is fixed, $t,\tau>0$, $t$ is fixed,
$\overline{h}=y-x$, $\Delta t=t-\tau$, $\overline{H}=(\overline{h},\Delta t)$,
$k$ is an integer, $k>l$, $\varepsilon>0$. Represent $u(x,t)$ in the form%

\begin{equation}
u(x,t)=\Delta_{\overline{H}}^{k}u+{\sum\limits_{i=1}^{k}}c_{j}u(x+i\overline
{h},t+i\Delta t)=\frac{\Delta_{\overline{H}}^{k}u}{||\overline{H}||^{l}%
}||\overline{H}||^{l}+{\sum\limits_{i=1}^{k}}c_{j}u(x+i\overline{h},t+i\Delta
t). \label{2.4}%
\end{equation}
From this we obtain%

\begin{equation}
|u(x,t)|\leq C\left\langle u\right\rangle _{R^{N}\times R_{+}^{1}}%
^{(l)}||\overline{H}||^{l}+{\sum\limits_{i=1}^{k}}c_{i}|u(x+i\overline
{h},t+i\Delta t)| \label{2.5}%
\end{equation}
with some constants $c_{i}$, $R_{+}^{1}=\{t\geq0\}$.

Raising this inequality to the power $p>1$ and applying elementary estimates,
we obtain%
\begin{equation}
|u(x,t)|^{p}\leq C\left(  \left\langle u\right\rangle _{R^{N}\times R_{+}^{1}%
}^{(l)}\right)  ^{p}||\overline{H}||^{pl}+C{\sum\limits_{i=1}^{k}}\left\vert
u\right\vert ^{p}(x+i\overline{h},t+i\Delta t). \label{2.6}%
\end{equation}
Integrating in $(y,\tau)$ over the cylinder $Q_{\varepsilon}(x,t)=\{(y,\tau
):|y-x|\leq\varepsilon,t\leq\tau\leq t+\varepsilon^{2}\}$, we get%

\begin{equation}
C\varepsilon^{N+2}|u(x,t)|^{p}\leq C\left(  \left\langle u\right\rangle
_{R^{N}\times R_{+}^{1}}^{(l)}\right)  ^{p}\varepsilon^{N+2+pl}+C_{i}%
{\sum\limits_{i=1}^{k}}{\int\limits_{Q_{i\varepsilon}(x,t)}}\left\vert
u\right\vert ^{p}(z,\theta)dzd\theta, \label{2.7}%
\end{equation}
where for each $i$ in the sum we made the change of the variables
$z=x+i(y-x)$, $\theta=t+i(\tau-t)$ and we took into account that for each $i$
we have $(y-x)=(z-x)/i$. If the norm in \eqref{2.3.2} is finite, we can
integrate just in $y$ over the ball $B_{\varepsilon}(x,t)=\{y:|y-x|\leq
\varepsilon\}$ and obtain%

\begin{equation}
C\varepsilon^{N}|u(x,t)|^{p}\leq C\left(  \left\langle u\right\rangle
_{R^{N}\times R_{+}^{1}}^{(l)}\right)  ^{p}\varepsilon^{N+pl}+C_{i}%
{\sum\limits_{i=1}^{k}}{\int\limits_{B_{i\varepsilon}(x,t+i\Delta t)}%
}\left\vert u\right\vert ^{p}(z,t+i\Delta t)dz, \label{2.7.1}%
\end{equation}

Estimate now the integrals over $Q_{i\varepsilon}(x,t)$ in \eqref{2.7} by the
integral over $R^{N}\times R_{+}^{1}$, divide both sides of \eqref{2.7} by
$C\varepsilon^{N+2}$ and take the roots of power $p$ from the terms of this
relation. As a result we obtain%

\[
|u(x,t)|\leq C\varepsilon^{l}\left\langle u\right\rangle _{R^{N}\times
R_{+}^{1}}^{(l)}+C\varepsilon^{-\frac{N+2}{p}}\left\Vert u\right\Vert
_{p,R^{N}\times R_{+}^{1}},
\]
or, taking supremum over $(x,t)\in R^{N}\times R_{+}^{1}$,%

\begin{equation}
|u|_{R^{N}\times R_{+}^{1}}^{(0)}\leq C\varepsilon^{l}\left\langle
u\right\rangle _{R^{N}\times R_{+}^{1}}^{(l)}+C\varepsilon^{-\frac{N+2}{p}%
}\left\Vert u\right\Vert _{p,R^{N}\times R_{+}^{1}}. \label{2.8}%
\end{equation}
Optimizing this inequality with respect to $\varepsilon>0$, or just taking
\[
\varepsilon=\left(  \left\Vert u\right\Vert _{p,R^{N}\times R_{+}^{1}%
}/\left\langle u\right\rangle _{R^{N}\times R_{+}^{1}}^{(l)}\right)
^{\frac{p}{pl+N+2}}%
\]
with $\left\langle u\right\rangle _{R^{N}\times R_{+}^{1}}^{(l)}\neq0$, we
finally obtain%

\begin{equation}
|u|_{R^{N}\times R_{+}^{1}}^{(0)}\leq C\left(  \left\langle u\right\rangle
_{R^{N}\times R_{+}^{1}}^{(l)}\right)  ^{\frac{N+2}{pl+N+2}}\left(  \left\Vert
u\right\Vert _{p,R^{N}\times R_{+}^{1}}\right)  ^{\frac{pl}{pl+N+2}}
\label{2.9}%
\end{equation}
that is exactly inequality \eqref{2.3.1}. If now $\left\langle u\right\rangle
_{R^{N}\times R_{+}^{1}}^{(l)}=0$ then from \eqref{2.8} after letting
$\varepsilon\rightarrow\infty$ it follows that $|u|_{R^{N}\times R_{+}^{1}%
}^{(0)}=0$ and so \eqref{2.9} is valid in this case also.

If the norm in \eqref{2.3.2} is finite, completely analogously to \eqref{2.9},
we have subsequently from \eqref{2.7.1}%
\[
|u(x,t)|\leq C\varepsilon^{l}\left\langle u\right\rangle _{R^{N}\times
R_{+}^{1}}^{(l)}+C\varepsilon^{-\frac{N}{p}}\sup_{t\in R_{+}^{1}}\left\Vert
u(\cdot,t)\right\Vert _{p,R^{N}},
\]%
\[
|u|_{R^{N}\times R_{+}^{1}}^{(0)}\leq C\varepsilon^{l}\left\langle
u\right\rangle _{R^{N}\times R_{+}^{1}}^{(l)}+C\varepsilon^{-\frac{N}{p}}%
\sup_{t\in R_{+}^{1}}\left\Vert u(\cdot,t)\right\Vert _{p,R^{N}},
\]
and, optimizing this inequality with respect to $\varepsilon>0$,
\[
|u|_{R^{N}\times R_{+}^{1}}^{(0)}\leq C\left(  \left\langle u\right\rangle
_{R^{N}\times R_{+}^{1}}^{(l)}\right)  ^{\frac{N}{pl+N}}\left(  \sup_{t\in
R_{+}^{1}}\left\Vert u(\cdot,t)\right\Vert _{p,R^{N}}\right)  ^{\frac
{pl}{pl+N}}%
\]
that is exactly inequality \eqref{2.3.3}.

Now in the case of general smooth domain $\Omega\neq R^{N}$ and $T<\infty$ the
lemma follows by an extension of a given function to $R^{N}\times R_{+}^{1}$
with the preserving of the corresponding norms.The lemma is proved.
\end{proof}

Now we can easily get the following more general assertion.

\begin{theorem}
\label{L2.2} Let $l_{1}$ be any positive number and let $l_{2}>l_{1}$ be a
positive noninteger. Let also $u(x,t)\in C^{l_{2},\frac{l_{2}}{2}}%
(\overline{\Omega}_{T})\cap L_{p}(\Omega_{T})$. Then%
\begin{equation}
|u|_{\overline{\Omega}_{T}}^{(l_{1})}\leq C\left(  |u|_{\overline{\Omega}_{T}%
}^{(l_{2})}\right)  ^{\omega}\left(  ||u||_{p,\overline{\Omega}_{T}}\right)
^{1-\omega},\quad\omega=\frac{pl_{1}+N+2}{pl_{2}+N+2}. \label{2.10}%
\end{equation}

If the parabolic norm%
\[
\sup_{0<t<T}\left\Vert u(\cdot,t)\right\Vert _{p,\Omega}<\infty
\]
is finite, then%
\begin{equation}
|u|_{\overline{\Omega}_{T}}^{(l_{1})}\leq C\left(  |u|_{\overline{\Omega}_{T}%
}^{(l_{2})}\right)  ^{\sigma}\left(  \sup_{0<t<T}\left\Vert u(\cdot
,t)\right\Vert _{p,\Omega}\right)  ^{1-\sigma},\quad\sigma=\frac{pl_{1}%
+N}{pl_{2}+N}. \label{2.10.1}%
\end{equation}

\end{theorem}

\begin{proof}
From \eqref{2.2} we have%
\[
|u|_{\overline{\Omega}_{T}}^{(l_{1})}\leq C\left(  |u|_{\overline{\Omega}_{T}%
}^{(l_{2})}\right)  ^{\frac{l_{1}}{l_{2}}}\left(  |u|_{\overline{\Omega}_{T}%
}^{(0)}\right)  ^{\frac{l_{2}-l_{1}}{l_{2}}}.
\]
Substituting in this inequality estimate \eqref{2.3.1} for $|u|_{\overline
{\Omega}_{T}}^{(0)}$ with $l=l_{2}$, we obtain%

\[
|u|_{\overline{\Omega}_{T}}^{(l_{1})}\leq C\left(  |u|_{\overline{\Omega}_{T}%
}^{(l_{2})}\right)  ^{\frac{l_{1}}{l_{2}}}\left[  \left(  |u|_{R^{N}\times
R_{+}^{1}}^{(l_{2})}\right)  ^{\frac{N+2}{pl_{2}+N+2}}\left(  \left\Vert
u\right\Vert _{R^{N}\times R_{+}^{1}}\right)  ^{\frac{pl_{2}}{pl_{2}+N+2}%
}\right]  ^{\frac{l_{2}-l_{1}}{l_{2}}}=
\]

\[
=C\left(  |u|_{\overline{\Omega}_{T}}^{(l_{2})}\right)  ^{\frac{pl_{1}%
+N+2}{pl_{2}+N+2}}\left(  \left\Vert u\right\Vert _{R^{N}\times R_{+}^{1}%
}\right)  ^{\frac{p(l_{2}-l_{1})}{pl_{2}+N+2}}%
\]
that is \eqref{2.10}. Inequality \eqref{2.10.1} is completely analogous on the
base of \eqref{2.3.3}. The Theorem is proved.
\end{proof}

By exactly the same arguments we have also an assertion for isotropic
H\"{o}lder spaces in "elliptic" case.

\begin{theorem}
Let $l_{1}$ be any positive number and let $l_{2}>l_{1}$ be a positive
noninteger. Let also $u(x)\in C^{l_{2}}(\overline{\Omega})\cap L_{p}(\Omega)$.
Then%
\begin{equation}
|u|_{\overline{\Omega}}^{(l_{1})}\leq C\left(  |u|_{\overline{\Omega}}%
^{(l_{2})}\right)  ^{\omega}\left(  ||u||_{p,\overline{\Omega}}\right)
^{1-\omega},\quad\omega=\frac{pl_{1}+N}{pl_{2}+N}. \label{2.11}%
\end{equation}

\end{theorem}

In the next section we give some simple application to an initial boundary
value problem for a quasilinear parabolic equation.


\begin{thebibliography}{9}                                                                                                %
\bibitem {Lun1}A. Lunardi, \textit{Interpolation theory}, Edizioni Della
Normale, Scuola Normale Superiore, Pisa. 2018. 199 p.

\bibitem {6}A. Lunardi, \textit{Analytic semigroups and optimal regularity in
parabolic problems}, Progress in Nonlinear Differential Equations and their
Applications. 16. Birkhauser Verlag, Basel. 1995. 424 p.

\bibitem {Triebel}H. Triebel, \textit{Theory of function spaces II}, Reprint
of the 1992 edition. Modern Birkhauser Classics. Birkhauser, Basel. 2010. 370 p.

\bibitem {7}O.A. Ladyzhenskaya, \textit{A Theorem on Multiplicators in
Nonhomogeneous Holder Spaces and Some of Its Applications} // Journal of
Mathematical Sciences (New York). v. 115, no. 6. 2003.  P.~ 2792--2802.

\bibitem {5}V.A. Solonnikov, \textit{Estimates for solutions of a
non-stationary linearized system of Navier--Stokes equations} // In: Boundary
value problems of mathematical physics. Part 1, Collection of articles, Trudy
Mat. Inst. Steklov., Nauka, Moscow--Leningrad. v. 70. 1964. P. 213--317.

\bibitem {Gol18}K.K. Golovkin, \textit{On equivalent normalizations of
fractional spaces} // In: Automatic programming, numerical methods and
functional analysis, Trudy Mat. Inst. Steklov. v. 66. 1962. P. 364--383.

\bibitem {8}V.A. Solonnikov, \textit{On the differential properties of the
solution of the first boundary-value problem for a non-stationary system of
Navier--Stokes equations} // Boundary value problems of mathematical physics.
Part 2, Collection of articles. Dedicated to the memory of Vladimir Andreevich
Steklov in connection with the centennial of his birth, Trudy Mat. Inst.
Steklov. v. 73. 1964. P. 221--291.

%\bibitem {9}

\end{thebibliography}
\end{document}